\documentclass[11pt,a4paper]{article}
\usepackage[T1]{fontenc}
\usepackage{lmodern,amsmath,amssymb,amsthm,mathtools,microtype,booktabs}
\usepackage[margin=27mm]{geometry}
\usepackage[colorlinks=true,linkcolor=blue,citecolor=blue,urlcolor=blue]{hyperref}
\hypersetup{pdftitle={Spectral analysis and certified evaluation of rank-two weighted LIL constants},
pdfauthor={Elina Moldavskaya},pdfsubject={Rank-two weighted LIL constants}}
\newtheorem{theorem}{Theorem}[section]
\newtheorem{lemma}[theorem]{Lemma}
\newtheorem{proposition}[theorem]{Proposition}

\theoremstyle{remark}
\newcommand{\dd}{\,\mathrm d}

\newcommand{\Var}{\operatorname{Var}}
\newcommand{\tr}{\operatorname{tr}}
\newcommand{\sgn}{\operatorname{sgn}}

\newcommand{\norm}[1]{\left\lVert#1\right\rVert}
\newcommand{\ip}[2]{\left\langle#1,#2\right\rangle}
\numberwithin{equation}{section}

\title{Constants in the Weighted Law of the Iterated Logarithm under Long-Range Dependence: Hermite Rank Two}
\author{Elina Moldavskaya\\[4pt]
\small Technion--Israel Institute of Technology, Israel\\
\small\href{mailto:elina.mol@technion.ac.il}{\texttt{elina.mol@technion.ac.il}}}
\date{}
\begin{document}
\maketitle
\begin{abstract}
At Hermite rank two, the constant in the weighted law of the iterated
logarithm under long-range dependence is represented as the largest
eigenvalue of a suitably normalized positive integral operator on the
unit interval. The same eigenvalue determines the exponential-moment
threshold and the logarithmic right-tail rate of the weighted
second-chaos limit, which in the unweighted case is the Rosenblatt law.
The weighted third spectral moment is evaluated in closed form.
Convergent two-sided spectral enclosures are obtained, and in the
weight-concentration limit the leading eigenvalue is characterized by
a scalar equation with a uniform geometric remainder. At the
unweighted memory boundary, the constant decays like the square root
of the distance to criticality. Its leading coefficient is determined
to $25$ decimal places with certified full-operator residual bounds.
After division by the square-root boundary factor, the
weight-concentration and memory limits commute. 
Their common coefficient differs from the unweighted boundary
coefficient and lies in the certified interval
$(1.370323114331,\,1.370323114332)$.
The joint asymptotic formula is established with a uniform
two-parameter remainder bound.
\end{abstract}


\section{Introduction}\label{sec:introduction}
The rank-two constant in a non-Gaussian law of the iterated logarithm is
the largest eigenvalue of a normalized second-chaos kernel. This connects
a pathwise limit theorem with the spectral description of its limiting
distribution. A variational characterization alone, however, does not
supply a quantitative evaluation or control the error of a finite matrix
approximation.

Our starting stochastic result is the weighted LIL in
\cite{MoldavskayaLIL}. The companion paper \cite{MoldavskayaGeneral}
studies global bounds, parameter expansions and weight concentration for
all Hermite ranks. The present paper treats the spectral questions specific
to rank two. The common definitions are recalled explicitly; none of the
companion paper's general-rank estimates depends on the numerical
certificates proved here.

For $0<\alpha<1/2$ and $2d+2\alpha<1$, let
\[
 (T_{\alpha,d}f)(x)
 =\int_0^1x^{-d/2}y^{-d/2}|x-y|^{-\alpha}f(y)\dd y.
\]
The constant is its leading eigenvalue divided by
$\sqrt{2\tr(T_{\alpha,d}^2)}$. Our results describe this spectral
quantity, its distributional interpretation, and its two parameter limits.

First, Section~\ref{sec:ranktwo} identifies the moment threshold and
logarithmic tail rate of the weighted second-chaos law.
Section~\ref{sec:traces} evaluates its third spectral moment and explains
why the Selberg product cannot be substituted for all higher cycle traces.
Second, Section~\ref{sec:interval-certificate} proves
\[
 C_{2,0}=\frac12\lambda_{\max}(K_{1/2}),\qquad
 \Lambda_2(\alpha,0)=C_{2,0}\sqrt{1-2\alpha}
                       [1+O(1-2\alpha)].
\]
It supplies an analytic bound for the remaining spectrum and a certified
interval for this coefficient. The operator at $\alpha=1/2$ is auxiliary:
the stochastic LIL is not asserted at that excluded endpoint.
Third, Section~\ref{sec:concentration} gives the fixed-space operator
representation, its Laguerre matrix and spectral error bounds. Theorem
\ref{thm:joint-endpoint} identifies the common coefficient of the
weight-concentration and memory limits, distinct from $C_{2,0}$.

\paragraph{Relation to earlier work.}
Mori--Oodaira already express unweighted LIL endpoints as variational
extrema and give the rank-two spectral interpretation
\cite[Corollary 3.2]{MO86}; their functional theorem treats nonlinear
Gaussian sequences \cite[Theorems 5--6]{MO87}. Veillette--Taqqu identify
the interval operator for the Rosenblatt law and study its distribution
\cite{VT13}; Leonenko--Pepelyshev develop numerical approximations
\cite{LP25}. These spectral characterizations, the elementary tail
argument, and the general residual principle are not claimed as new.
The present analysis develops explicit weighted formulas, parameter-limit
results, and reproducible full-operator error certificates.

\paragraph{Motivation from learning with dependent observations.}
Weighted empirical risk minimization under long-range dependence leads to
sharp pathwise constants and cluster sets that depend on the sample
weights; see \cite{MoldavskayaML}. This provides an application-oriented
motivation for quantitative evaluation of weighted LIL constants.
The present paper treats the rank-two spectral problem. It does not
assert a finite-sample learning guarantee or a 25-digit evaluation for
arbitrary interior parameters.

\section{Model and normalization}\label{sec:model}
Let $(X_t)$ be a centered stationary standard Gaussian sequence with
\[
 \rho(h)\sim c h^{-\alpha}L_0(h),\qquad
 f_X(\lambda)\sim c_f|\lambda|^{\alpha-1}L_0(1/|\lambda|),
\]
where $L_0$ is eventually positive and slowly varying and $c,c_f>0$.
We retain both conditions, with the convention
$\rho(h)=\int_{-\pi}^{\pi}e^{ih\lambda}f_X(\lambda)\dd\lambda$.
The density is bounded on compact sets away from zero and belongs to
$W^{1,1}([-\pi,-\delta]\cup[\delta,\pi])$ for every $0<\delta<\pi$.
The centered function $G\in L^2(\mathbb R,\phi)$ has Hermite rank two,
where $\phi$ is standard Gaussian measure:
\[
 G(x)=\sum_{q=2}^{\infty}\frac{J_q}{q!}H_q(x),\qquad J_2\ne0.
\]
For weights $a_t=t^{-d}\ell_a(t)$, the factor $\ell_a$ is eventually
positive, slowly varying and continuously differentiable, with
$x\ell_a'(x)/\ell_a(x)\to0$. Put
\[
 S_n=\sum_{t=1}^n a_tG(X_t),\qquad A_n^2=\Var(S_n),
 \qquad 0<\alpha<\tfrac12,\quad 2d+2\alpha<1.
\]
Under these hypotheses, the first paper gives
\[
 \limsup_{n\to\infty}\frac{|S_n|}{A_n(2\log\log n)}
 =\Lambda_2(\alpha,d)\quad\text{almost surely}.
\]
We take this stochastic theorem as input. The deterministic operator
statements below do not require a second proof of it. The case $d=0$
includes constant weights and admissible nonconstant slowly varying
weights, since their limiting profile is one after variance normalization.

Write $\eta=1-d$, $w(x)=x^{-d}$, $W=1/\eta$, and
$\mu_d(\dd x)=w(x)\dd x/W$. Define
\begin{align}
 I_{2,\alpha,d}
 &=\iint x^{-d}y^{-d}|x-y|^{-2\alpha}\dd x\dd y
 =\frac{2B(1-d,1-2\alpha)}{2-2d-2\alpha},\label{quad:eq:I}\\
 (R_\alpha f)(x)&=\int_0^1|x-y|^{-\alpha}f(y)\mu_d(\dd y),
 &Z_2&=I_{2,\alpha,d}/W^2.\label{quad:eq:R}
\end{align}
The beta identity follows by $y=xt$ on $0<y<x<1$.
The rank-two specialization of the exact variational formula in
\cite[Proposition 3.1]{MoldavskayaGeneral} is
\[
 \Lambda_2(\alpha,d)=\frac{D_2}{\sqrt{2Z_2}},\qquad
 D_2=\sup_{\|f\|_{L^2(\mu_d)}=1}\langle f,R_\alpha f\rangle.
\]
We use this identity as an input from that manuscript; its general-rank
proof is not repeated here. The normalization is $2\|Q_1\|_2^2=1$ for the second-chaos
kernel in \cite{MoldavskayaLIL}. The sign of $J_2$ disappears in the
absolute LIL endpoint; it is retained when orienting the limiting law below.

\section{Rank two: weighted law, Rosenblatt specialization, and tails}\label{sec:ranktwo}
Let $m=2$ and define the positive Hilbert--Schmidt operator
\[
 (Tf)(x)=\int_0^1x^{-d/2}y^{-d/2}|x-y|^{-\alpha}f(y)\dd y.
\]
Let $\mu_1\ge\mu_2\ge\cdots>0$ be its eigenvalues and put
\[
 a_j=\frac{\mu_j}{\sqrt{2I_{2,\alpha,d}}},\qquad
 2\sum_ja_j^2=1,\qquad a_1=\Lambda_2(\alpha,d).
\]
For $m=2$, multiplication by $\sqrt{w/W}$ is a unitary map from
$L^2(\mu_d)$ to $L^2(0,1)$ and conjugates $R_\alpha$ to $T/W$.
Since $Z_2=I_{2,\alpha,d}/W^2$, the variational identity of Section~\ref{sec:model}
gives the displayed identity for $a_1$. The Riesz factorization of the
second-chaos kernel, also recorded in \cite{MoldavskayaLIL}, identifies
the sign-oriented, variance-one limit with
\begin{equation}\label{spectral:eq:law}
 Y_{\alpha,d}:=\sgn(J_2)I_2(Q_1)
 \ \stackrel{\mathrm{law}}=\ \sum_{j\ge1}a_j(Z_j^2-1),
 \qquad Z_j\ \hbox{independent }N(0,1),
\end{equation}
with $L^2$ and almost-sure convergence. The sign orientation matters
when describing a one-sided tail; the original LIL uses an absolute value.

\paragraph{Precisely when the ordinary Rosenblatt law occurs.}
For $d=0$, $T=K_\alpha$ and $Y_{\alpha,0}$ is the variance-one Rosenblatt
variable of parameter $H_R=1-\alpha$. Its expansion coefficient
$a_1$ is exactly $\Lambda_2(\alpha,0)$; this is the operator in
Veillette--Taqqu, Proposition 3.1 \cite{VT13}.
For $d\ne0$, \eqref{spectral:eq:law} instead identifies the weighted
second-chaos limit through the spectrum of $T_{\alpha,d}$. The ordinary
Rosenblatt identification here applies to $d=0$.

\begin{proposition}\label{spectral:prop:tail}
For all admissible rank-two parameters,
\begin{gather}
 \mathbb E e^{tY_{\alpha,d}}<\infty
 \quad\Longleftrightarrow\quad
 t<\frac{1}{2\Lambda_2(\alpha,d)},\qquad t\in\mathbb R,
 \label{spectral:eq:mgf-domain}\\
 \lim_{x\to\infty}\frac1x\log\mathbb P(Y_{\alpha,d}>x)
 =-\frac{1}{2\Lambda_2(\alpha,d)}.\label{spectral:eq:right-tail}
\end{gather}
The same logarithmic limit holds for $\mathbb P(|Y_{\alpha,d}|>x)$.
\end{proposition}
\begin{proof}
For $t<(2a_1)^{-1}$, independence gives
\begin{equation}\label{spectral:eq:log-mgf}
 \log\mathbb E e^{tY_{\alpha,d}}
 =\sum_{j\ge1}\left[-ta_j-\frac12\log(1-2ta_j)\right].
\end{equation}
The summands are $O(a_j^2)$ in the tail of this sum. To justify the
expectation for the infinite series, choose $p>1$ with $pt<(2a_1)^{-1}$
when $t>0$; for $t\le0$ any fixed $p>1$ works. The same convergent
product at $pt$ gives uniform integrability of the finite exponentials.
For $t\ge(2a_1)^{-1}$, separate the first term of \eqref{spectral:eq:law};
its exponential moment is infinite, and independence gives divergence
for the whole sum.

For $0<t<(2a_1)^{-1}$, Chernoff's inequality gives the upper logarithmic
rate $-t$, hence at most $-(2a_1)^{-1}$. Write
$Y_{\alpha,d}=a_1(Z_1^2-1)+W$, where $W$ is independent of $Z_1$.
Choose $M$ with $\mathbb P(W\ge-M)>0$. Then
\[
 \mathbb P(Y_{\alpha,d}>x)\ge
 \mathbb P(W\ge-M)\mathbb P\left(Z_1^2>
                \frac{x+M+a_1}{a_1}\right),
\]
which gives the matching lower logarithmic rate by the normal tail.
Finally, \eqref{spectral:eq:mgf-domain} holds for every negative $t$; applying
Chernoff to $-Y_{\alpha,d}$ yields
$\limsup x^{-1}\log\mathbb P(Y_{\alpha,d}<-x)=-\infty$.
\end{proof}

For $d=0$, \cite[Corollary 4.5]{VT13} gives the stronger shifted-tail
ratio $\mathbb P(Y>x+u)/\mathbb P(Y>x)\to e^{-u/(2a_1)}$.
Proposition~\ref{spectral:prop:tail} is an elementary application of positive
second-chaos spectral theory to the weighted kernel, not a claim of
a new general tail theorem. It supplies a probabilistic interpretation of the LIL constant.

\section{Evaluated spectral moments and the transform}\label{sec:traces}
Let $\mathcal A=T/\sqrt{2I_{2,\alpha,d}}$. In the disk
$|t|<(2\Lambda_2)^{-1}$, \eqref{spectral:eq:log-mgf} becomes
\begin{equation}\label{spectral:eq:trace-series}
 \log\mathbb E e^{tY_{\alpha,d}}
 =\sum_{r=2}^{\infty}\frac{2^{r-1}}r
       t^r\tr(\mathcal A^r),\qquad
 \kappa_r(Y_{\alpha,d})=2^{r-1}(r-1)!\tr(\mathcal A^r).
\end{equation}
Equivalently the moment generating function is
$\det{}_2(I-2t\mathcal A)^{-1/2}$. The regularized determinant is
needed since Hilbert--Schmidt structure, rather than trace class,
is guaranteed. The determinant $\det{}_2(I-2t\mathcal A)$ has its
first positive zero at $(2\Lambda_2)^{-1}$. Naming this determinant
alone is not a closed-form evaluation of that zero.

\begin{proposition}[An evaluated weighted third spectral moment]
Put $a=1-d$ and $c=-\alpha/2$. Then
\begin{equation}\label{spectral:eq:selberg3}
 \boxed{\tr(T^3)=
 \prod_{j=0}^{2}
 \frac{\Gamma(a+jc)\Gamma(1+jc)\Gamma(1+(j+1)c)}
      {\Gamma(a+1+(2+j)c)\Gamma(1+c)}.}
\end{equation}
Consequently the skewness of the sign-oriented weighted limit is
\[
 \kappa_3(Y_{\alpha,d})
 =\frac{8\tr(T^3)}{(2I_{2,\alpha,d})^{3/2}}.
\]
At $d=0$, the expression reduces to
\begin{equation}\label{spectral:eq:beta3}
 \tr(K_\alpha^3)=\frac{6B(1-\alpha,1-\alpha)}
                   {(2-3\alpha)(3-3\alpha)}.
\end{equation}
\end{proposition}
\begin{proof}
The trace integral is
\[
 \tr(T^3)=\int_{[0,1]^3}
 (xyz)^{-d}|x-y|^{-\alpha}|y-z|^{-\alpha}|z-x|^{-\alpha}
 \dd x\dd y\dd z.
\]
This is precisely Selberg's integral $S_3(a,1,c)$, which gives
\eqref{spectral:eq:selberg3}; see \cite[Eq.~5.14.4, with no extra monomial]{DLMF}.
The convergence condition is $c>-\min(1/3,a/2,1/2)$.
It holds because $\alpha<1/2$ and $d+\alpha<1/2$, hence $a>\alpha$.
The trace identity follows first for bounded truncations of the
nonnegative kernel and then by monotone convergence and convergence
in Hilbert--Schmidt norm.
For $d=0$, order the three variables, write their span as $r$ and
the middle relative position as $u$, and integrate the left endpoint.
The result is
$6\int_0^1r^{1-3\alpha}(1-r)\dd r
\int_0^1[u(1-u)]^{-\alpha}\dd u$, yielding \eqref{spectral:eq:beta3}.
\end{proof}

\paragraph{Why the direct Selberg summation fails at the next order.}
For $r\ge2$ the trace integrand is a cycle:
\[
 \tr(T^r)=\int_{[0,1]^r}\prod_{i=1}^rx_i^{-d}
       \prod_{i=1}^r|x_i-x_{i+1}|^{-\alpha}\dd x_1\cdots\dd x_r,
 \qquad x_{r+1}=x_1.
\]
At $r=3$ the cycle is the complete graph. At $r=4$ the Selberg
integrand has two additional factors,
$|x_1-x_3|^{-\alpha}|x_2-x_4|^{-\alpha}$.
Both are strictly greater than one almost everywhere on the unit cube.
The Selberg integral is finite in the present parameter region, so
\[
 \tr(T^4)<S_4(1-d,1,-\alpha/2).
\]
Thus one cannot insert the Selberg gamma product as every coefficient
of \eqref{spectral:eq:trace-series}. The exact third moment does not determine
the nearest singularity of the whole transform. Other identities for
cycle integrals or a summation of their generating function remain
possible, but neither has been established here.

\section{The unweighted memory boundary and a certified interval}
\label{sec:interval-certificate}
Write $K_\alpha f(x)=\int_0^1|x-y|^{-\alpha}f(y)\dd y$ and
$K_*=K_{1/2}$. The auxiliary operator $K_*$ is bounded, compact and
positive, although $\tr(K_*^2)=\infty$. Indeed, deletion of
$|x-y|<\delta$ gives a Hilbert--Schmidt kernel and changes the operator
norm by at most $4\sqrt\delta$, by Schur's test. Positivity follows from
the positive Laplace mixture of the positive definite kernels
$e^{-t|x-y|}$. The strictly positive kernel makes the leading
eigenvalue simple with a positive eigenfunction. Denote it by $\lambda_*$.

\begin{proposition}[Unweighted boundary coefficient]\label{prop:interval-boundary}
As $\alpha\uparrow1/2$ through $0<\alpha<1/2$,
\[
 \Lambda_2(\alpha,0)=C_{2,0}\sqrt{1-2\alpha}
 [1+O(1-2\alpha)],\qquad C_{2,0}=\lambda_*/2.
\]
This is an asymptotic of the subcritical LIL constant; it is not a LIL
at $\alpha=1/2$.
\end{proposition}
\begin{proof}
For $\alpha$ in a neighborhood of $1/2$, the local Schur estimate for
$|x-y|^{-b}|\log|x-y||$, with $b<1$, gives
$\|K_\alpha-K_*\|=O(|\alpha-1/2|)$.
The variational principle yields the same Lipschitz bound for the leading
eigenvalue. With $\varepsilon=1-2\alpha$ the exact normalization is
\[
 \Lambda_2(\alpha,0)
 =\sqrt{\frac{(1-2\alpha)(1-\alpha)}2}\lambda_{\max}(K_\alpha)
 =\frac{\sqrt\varepsilon}2\sqrt{1+\varepsilon}
     [\lambda_*+O(\varepsilon)].
\]
This proves the result.
\end{proof}

\begin{lemma}[An analytic bound for the remaining spectrum]
\label{lem:analytic-deflation}
The second eigenvalue satisfies
\[
 \lambda_2(K_*)\le a_{\rm S}:=\frac{9+\sqrt3}{8}.
\]
\end{lemma}
\begin{proof}
Let $P$ be the orthogonal projection onto $\mathbf1$ in $L^2(0,1)$.
On $\mathbf1^\perp$, the quadratic forms of $K_*$ and $K_*-2P$
coincide. Min--max and Schur's test therefore give
\[
 \lambda_2(K_*)\le\|K_*-2P\|
 \le\sup_x\int_0^1\bigl||x-y|^{-1/2}-2\bigr|\dd y.
\]
Put $H(t)=\int_0^t|u^{-1/2}-2|\dd u$. Direct integration gives
\[
 H(t)=\begin{cases}
 2\sqrt t-2t,&0\le t\le1/4,\\
 2t-2\sqrt t+1,&1/4\le t\le1.
 \end{cases}
\]
The row integral is $H(x)+H(1-x)$ and is symmetric about $1/2$.
For $1/4\le x\le1/2$ it is at most $3-\sqrt3<a_{\rm S}$.
For $0\le x\le1/4$, concavity gives the chord bound
$\sqrt{1-x}\ge1-(4-2\sqrt3)x$. Hence the row integral is at most
\[
 3-4x+2\sqrt x-2\sqrt{1-x}
 \le1+2\sqrt x-4(\sqrt3-1)x
 \le1+\frac1{4(\sqrt3-1)}=a_{\rm S}.
\]
The last inequality is completion of the square in $\sqrt x$.
The multiplier 2 in $K_*-2P$ need not be a Rayleigh quotient.
\end{proof}

\begin{proposition}[Full-operator residual certificate]
\label{prop:interval-residual}
For nonzero real $g\in L^2(0,1)$ define
\[
 \theta=\frac{\langle g,K_*g\rangle}{\|g\|_2^2},\qquad
 \sigma^2=\frac{\|K_*g-\theta g\|_2^2}{\|g\|_2^2}.
\]
If $\theta>a_{\rm S}$, then
\begin{equation}\label{eq:interval-residual}
 \theta\le\lambda_*\le\theta+
                   \frac{\sigma^2}{\theta-a_{\rm S}}.
\end{equation}
For any real shift $\tau$, the squared residual with $\theta$ is at most
$\|K_*g-\tau g\|_2^2/\|g\|_2^2$.
\end{proposition}
\begin{proof}
For every spectral value $t$ of $K_*$,
$(t-\lambda_*)(t-a_{\rm S})\ge0$. Integrating this inequality against
the spectral probability measure of $g/\|g\|_2$ gives
$\sigma^2\ge(\lambda_*-\theta)(\theta-a_{\rm S})$.
The lower bound is the variational principle. The shift assertion follows
from the orthogonality of $K_*g-\theta g$ and $g$.
This is the classical residual argument, specialized using the explicit
spectral bound of Lemma~\ref{lem:analytic-deflation}.
\end{proof}

\begin{theorem}[Certified unweighted boundary value]
\label{thm:interval-value}
The following outward-rounded bounds hold:
\[
 \begin{aligned}
 2.682918382150264832337066716763&<\lambda_*\\
 &<2.682918382150264832337066725454.
 \end{aligned}
\]
Consequently
\[
 \begin{aligned}
 1.341459191075132416168533358381&<C_{2,0}\\
 &<1.341459191075132416168533362727.
 \end{aligned}
\]
In particular, $\lambda_*$ rounds to
$2.6829183821502648323370667$ at 25 decimal places.
\end{theorem}
\begin{proof}
Use the rational 128-function trial and the exact full-residual
construction in Appendix~\ref{app:interval-certificate}.
Directed interval arithmetic gives a Rayleigh interval
$[q_-,q_+]$, an upper bound $s_+$ for $\sigma^2$, and an upper bound
$a_+$ for $a_{\rm S}$, with $q_->a_+$. Proposition
\ref{prop:interval-residual} encloses the eigenvalue between
\[
 q_-,\qquad q_++\frac{s_+}{q_--a_+}.
\]
The exact dyadic endpoints and rational trial vector are included in the
source of this article, as specified in Appendix~\ref{app:certificate-records}. Integer floor and ceiling at the displayed decimal
scale give the stated endpoints. Their unrounded width is less than
$8.690\cdot10^{-27}$. Division by two gives the coefficient bounds.
The proof of the integration error bound is included in the appendix;
no agreement of successive finite matrices is used as an error estimate.
\end{proof}

For comparison, the finite third trace is
\[
 \tr(K_*^3)=6B(1/2,1/2)\int_0^1r^{-1/2}(1-r)\dd r=8\pi.
\]
It follows either by ordering the three coordinates, or by extending
\eqref{spectral:eq:beta3} with positive Laplace-mixture approximants.
These approximants increase in operator order and in their positive
three-cycle integrals, justifying both limits. Thus the alternative
spectral bound $(8\pi-q_-^3)^{1/3}$ remains valid even though
$\tr(K_*^2)$ is infinite. For the same trial and residual, the analytic
bound $a_{\rm S}$ reduces the residual correction by approximately
$34\%$ compared with this third-trace bound. This is a comparison of
certificate widths, not a closed-form evaluation of $\lambda_*$.

\section{Concentrating weights and spectral evaluation}\label{sec:concentration}
Let $0<\alpha<1/2$, $\eta=1-d$, $\kappa=\eta-\alpha>1/2$, and
$\nu(\dd u)=e^{-u}\dd u$ on $[0,\infty)$. Set
\begin{align}
 k_{\alpha,\kappa}(u,v)
 &=\left[2\kappa\sinh\left(\frac{|u-v|}{2\kappa}\right)\right]^{-\alpha},
 &k_{\alpha,\infty}(u,v)&=|u-v|^{-\alpha},\label{eq:hyperbolic-kernel}\\
 \mathcal Z_{s,\kappa}
 &=\kappa^{1-s}B(\kappa+s/2,1-s),
 &\mathcal Z_{s,\infty}&=\Gamma(1-s),\quad s=2\alpha.
 \label{eq:exponential-Z}
\end{align}
Let $S_{\alpha,\kappa}$ be the integral operator with kernel
$k_{\alpha,\kappa}$ on $L^2(\nu)$, and write
$\mathcal D_{2,\kappa}=\lambda_{\max}(S_{\alpha,\kappa})$.

\begin{proposition}[Fixed-space representation]\label{thm:concentration}
For all admissible rank-two parameters,
\begin{equation}\label{eq:concentration-exact}
 \Lambda_2(\alpha,d)=
 \frac{\mathcal D_{2,\kappa}}{\sqrt{2\mathcal Z_{2\alpha,\kappa}}}.
\end{equation}
For each fixed $0<\alpha<1/2$, as $d\to-\infty$,
\[
 \Lambda_2(\alpha,d)=\Lambda_2^{\exp}(\alpha)
 [1+O((1-d)^{-2})],\qquad
 \Lambda_2^{\exp}(\alpha)=
 \frac{\lambda_{\max}(S_{\alpha,\infty})}{\sqrt{2\Gamma(1-2\alpha)}}>0.
\]
\end{proposition}
\begin{proof}
The map
\[
 f(e^{-u/\kappa})=(\kappa/\eta)^{1/2}e^{\alpha u/(2\kappa)}h(u)
\]
is unitary from $L^2(\nu)$ onto $L^2(\mu_d)$, since
$\mu_d(\dd x)=(\eta/\kappa)e^{-\eta u/\kappa}\dd u$.
Using
\[
 |e^{-u/\kappa}-e^{-v/\kappa}|^{-\alpha}
 =\kappa^\alpha e^{\alpha(u+v)/(2\kappa)}k_{\alpha,\kappa}(u,v)
\]
in the quadratic form gives $D_2=\eta\kappa^{\alpha-1}
\mathcal D_{2,\kappa}$. Similarly
$Z_2=\eta^2\kappa^{2\alpha-2}\mathcal Z_{2\alpha,\kappa}$.
The factors cancel in the normalized quotient.
For independent $U,V$ with law $\nu$, $|U-V|$ has density $e^{-t}$.
Integration followed by $z=e^{-t/\kappa}$ gives
\eqref{eq:exponential-Z}.

The elementary inequality $1\le\sinh z/z\le e^{z^2/6}$ gives
\begin{equation}\label{eq:hyperbolic-difference}
 0\le t^{-\alpha}-[2\kappa\sinh(t/(2\kappa))]^{-\alpha}
 \le\frac{\alpha}{24\kappa^2}t^{2-\alpha}.
\end{equation}
Consequently
\begin{align}
 \|k_{\alpha,\infty}-k_{\alpha,\kappa}\|_{L^2(\nu\otimes\nu)}
 &\le\frac{\alpha\sqrt{\Gamma(5-2\alpha)}}{24\kappa^2},
 \label{eq:kernel-rate}\\
 0\le\Gamma(1-s)-\mathcal Z_{s,\kappa}
 &\le\frac{s\Gamma(3-s)}{24\kappa^2}.\label{eq:normalization-rate}
\end{align}
The first inequality bounds the difference of the leading eigenvalues.
Positivity of a maximizing function gives the correct order of those
eigenvalues. Their limiting value is positive, as the constant trial has
Rayleigh value $\Gamma(1-\alpha)$. The normalization has a positive
limit as well. Substitution proves the assertion. The companion paper
\cite{MoldavskayaGeneral} proves the corresponding statement for all ranks.
\end{proof}

\subsection{Rank two: an explicit matrix generating function}
Now let $m=2$ and $0<\alpha<1/2$. For
$\kappa=1-d-\alpha\in(1/2,\infty]$, let $S_{\alpha,\kappa}$
be the operator on $L^2(\nu)$ with kernel
$k_{\alpha,\kappa}$. It is positive and Hilbert--Schmidt; at finite
$\kappa$ positivity also follows from the preceding isometry.
Its largest eigenvalue is $\mathcal D_{2,\kappa}$.

Let $\mathcal L_n(u)=L_n^{(0)}(u)$ be the standard Laguerre
polynomials, an orthonormal basis of $L^2(\nu)$, and put
$A_{ij}^{(\kappa)}=\langle\mathcal L_i,
S_{\alpha,\kappa}\mathcal L_j\rangle_{L^2(\nu)}$.
Define the scalar Laplace transform
\begin{equation}\label{eq:kernel-Laplace}
 \ell_{\alpha,\kappa}(c)=
 \begin{cases}
 \kappa^{1-\alpha}B(\kappa c+\alpha/2,1-\alpha),&\kappa<\infty,\\
 \Gamma(1-\alpha)c^{\alpha-1},&\kappa=\infty,
 \end{cases}\qquad \Re c>0.
\end{equation}

\begin{proposition}[Laguerre matrix]
\label{prop:Laguerre}
For $|z|,|w|<1$,
\begin{equation}\label{eq:Laguerre-generating}
 \sum_{i,j\ge0} A_{ij}^{(\kappa)}z^iw^j
 =\frac{\ell_{\alpha,\kappa}((1-z)^{-1})
              +\ell_{\alpha,\kappa}((1-w)^{-1})}{2-z-w}.
\end{equation}
In particular, the limiting matrix has the explicit generating function
\begin{equation}\label{eq:Laguerre-limit}
 \boxed{\displaystyle
 \sum_{i,j\ge0} A_{ij}^{(\infty)}z^iw^j
 =\Gamma(1-\alpha)
 \frac{(1-z)^{1-\alpha}+(1-w)^{1-\alpha}}{2-z-w}.}
\end{equation}
Thus every finite compression is given by a finite calculation.
\end{proposition}
\begin{proof}
The classical generating function \cite[Eq.~18.12.13]{DLMF} is
\[
 \sum_{n\ge0}\mathcal L_n(u)z^n
 =(1-z)^{-1}\exp[-uz/(1-z)].
\]
For a translation kernel $k(|u-v|)$, splitting the quadrant at
$u=v$ gives
\[
 \iint e^{-au-bv}k(|u-v|)\dd u\dd v
 =\frac{\ell(a)+\ell(b)}{a+b},\qquad \Re a,\Re b>0,
\]
where $\ell(c)=\int_0^\infty e^{-ct}k(t)\dd t$.
For the hyperbolic kernel the substitution $z=e^{-t/\kappa}$
gives \eqref{eq:kernel-Laplace}. Use $a=(1-z)^{-1}$,
$b=(1-w)^{-1}$ to obtain \eqref{eq:Laguerre-generating}.
These manipulations hold first near zero. The matrix is square
summable, so its series is absolutely convergent on the bidisk;
analytic continuation proves the stated domain.
\end{proof}

For implementation, write
$q_n^{(\kappa)}=[z^n]\ell_{\alpha,\kappa}((1-z)^{-1})$.
Then
\begin{align}
 A_{00}^{(\kappa)}&=q_0^{(\kappa)},&
 A_{n0}^{(\kappa)}=A_{0n}^{(\kappa)}
 &=\tfrac12(A_{n-1,0}^{(\kappa)}+q_n^{(\kappa)}),\quad n\ge1,
 \label{eq:Laguerre-boundary}\\
 A_{ij}^{(\kappa)}&=\tfrac12(A_{i-1,j}^{(\kappa)}
                                      +A_{i,j-1}^{(\kappa)}),
 &&i,j\ge1.\label{eq:Laguerre-interior}
\end{align}
At infinity,
$q_n^{(\infty)}=\Gamma(1-\alpha)(-1)^n\binom{1-\alpha}{n}$.
At finite $\kappa$, for $n\ge1$,
\[
 q_n^{(\kappa)}
 =\sum_{j=1}^n\binom{n-1}{j-1}
      \frac{\ell_{\alpha,\kappa}^{(j)}(1)}{j!}.
\]
The derivatives are finite beta derivatives. The matrix entries need
not all be positive; positivity here is positivity of the operator.

\subsection{Convergent two-sided spectral bounds}
The first two useful trace powers are explicitly available. Set
\[
 Q_{2,\kappa}=\mathcal Z_{2\alpha,\kappa},\qquad
 Q_{3,\infty}
 =\frac{\Gamma(1-\alpha)^2\Gamma(1-3\alpha/2)}
        {\Gamma(1-\alpha/2)}.
\]
At finite $\kappa$ put
\begin{equation}\label{eq:exponential-trace-three}
 Q_{3,\kappa}=Q_{3,\infty}\kappa^{3-3\alpha}
 \prod_{j=0}^{2}
 \frac{\Gamma(\kappa+(1-j/2)\alpha)}
      {\Gamma(\kappa+1-j\alpha/2)}.
\end{equation}

\begin{proposition}[Finite matrix enclosures]
\label{prop:spectral-enclosures}
For $j=2,3$ and $\kappa\in(1/2,\infty]$,
$\tr(S_{\alpha,\kappa}^j)=Q_{j,\kappa}$, with
$Q_{2,\infty}=\Gamma(1-2\alpha)$.
Let $\theta_1^{(N)}\ge\cdots\ge\theta_N^{(N)}\ge0$ be the
eigenvalues of the $N\times N$ principal Laguerre compression.
Then, with an empty sum interpreted as zero,
\begin{align}
 \theta_1^{(N)}\le\lambda_{\max}(S_{\alpha,\kappa})
 &\le U_N:=\min_{j\in\{2,3\}}
 \left(Q_{j,\kappa}-\sum_{i=2}^N(\theta_i^{(N)})^j\right)^{1/j},
 \label{eq:spectral-enclosure}\\
 \theta_1^{(N)}\uparrow\lambda_{\max}(S_{\alpha,\kappa}),
 &\qquad U_N\downarrow\lambda_{\max}(S_{\alpha,\kappa}).
 \label{eq:spectral-enclosure-convergence}
\end{align}
Dividing both endpoints by $\sqrt{2Q_{2,\kappa}}$ gives convergent
two-sided bounds for $\Lambda_2(\alpha,d)$, or for
$\Lambda_2^{\exp}(\alpha)$ when $\kappa=\infty$.
\end{proposition}
\begin{proof}
The second trace is the squared kernel norm. At finite $\kappa$,
the fixed-space isometry gives
$\tr(S_{\alpha,\kappa}^3)=\kappa^{3-3\alpha}\tr(T^3)$,
where $T$ is the weighted interval operator of
Section~\ref{sec:ranktwo} and $\eta=\kappa+\alpha$.
Substitute \eqref{spectral:eq:selberg3}; simplifying its weight-free
gamma factors gives \eqref{eq:exponential-trace-three}.
Equation \eqref{eq:kernel-rate} gives Hilbert--Schmidt convergence
to $S_{\alpha,\infty}$. This implies convergence of the third
traces. The gamma-ratio asymptotic in
\eqref{eq:exponential-trace-three} yields $Q_{3,\infty}$.
Equivalently, this is the three-variable Laguerre--Selberg
integral \cite[Eq.~5.14.5, with no extra monomial]{DLMF}.

List the eigenvalues of $S_{\alpha,\kappa}$ in decreasing order
as $\lambda_i\ge0$. The min--max principle gives
$\theta_i^{(N)}\le\lambda_i$ for $i\le N$. Hence
\[
 \lambda_1^j=Q_{j,\kappa}-\sum_{i\ge2}\lambda_i^j
 \le Q_{j,\kappa}-\sum_{i=2}^N(\theta_i^{(N)})^j.
\]
This proves the enclosure. The nested Laguerre subspaces are dense,
so each fixed compressed eigenvalue increases to $\lambda_i$.
The sums in \eqref{eq:spectral-enclosure} increase to
$\sum_{i\ge2}\lambda_i^j$, by monotone convergence and min--max.
This proves both limits in \eqref{eq:spectral-enclosure-convergence}.
\end{proof}

These are exact inequalities for exact matrix eigenvalues. A numerical
implementation must also control the error in the finite matrix and
its eigensolver before reporting certified decimal endpoints. The
enclosures give a convergent evaluation procedure, not a finite
closed-form value of the largest eigenvalue. In particular, the
generating function \eqref{eq:Laguerre-limit} is a closed formula
for the matrix, not for its spectrum.

\subsection{A scalar equation with a uniform geometric remainder}
For the exponential limit, write $S_\alpha=S_{\alpha,\infty}$.
The operator remains positive and compact for $0<\alpha<1$, although
it is no longer Hilbert--Schmidt at $\alpha=1/2$.
Indeed, cutting out $|u-v|<\delta$ leaves a bounded, hence
Hilbert--Schmidt, kernel on the probability space. Schur's test bounds
the norm of the removed part by $2\delta^{1-\alpha}/(1-\alpha)$.
The same test applied to logarithmic derivatives of the kernel shows
that $\alpha\mapsto S_\alpha$ is locally analytic in operator norm
on $0<\alpha<1$. For small distances use
$t^{-b}|\log t|^j$ with $b<1$; for distances exceeding one use
the bounded function $t^{-a}(\log t)^j$ with $a>0$.

The third trace formula also holds at $\alpha=1/2$.
One justification is to approximate $t^{-\alpha}$ by
\[
 k_R(t)=\frac1{\Gamma(\alpha)}
             \int_0^R x^{\alpha-1}e^{-xt}\dd x.
\]
These give positive trace-class operators increasing in the operator
order to $S_\alpha$. Their third traces are the corresponding
three-cycle integrals. Monotone convergence of both eigenvalues
and positive integrands gives the trace identity whenever the
three-cycle integral is finite, in particular for $\alpha\le1/2$.
Thus $S_{1/2}$ belongs to the Schatten class of order three.

\begin{proposition}[An isolated scalar root]\label{prop:scalar-root}
Let $0<\alpha\le1/2$, $a_0=\Gamma(1-\alpha)$, and
\[
 Q_3=\frac{\Gamma(1-\alpha)^2\Gamma(1-3\alpha/2)}
                  {\Gamma(1-\alpha/2)},\qquad
 r_\alpha=\left(\frac{Q_3}{a_0^3}-1\right)^{1/3}.
\]
Let $P$ be the orthogonal projection onto the constant function
$\mathbf1$ in $L^2(\nu)$, $C=(I-P)S_\alpha(I-P)$ on
$\mathbf1^\perp$, and $b=(I-P)S_\alpha\mathbf1$.
Then
\[
 \norm C\le a_0r_\alpha,\qquad
 0<r_\alpha\le r_*:=
 \left(\frac{\Gamma(1/4)}{\sqrt\pi\Gamma(3/4)}-1\right)^{1/3}<1.
\]
Put $\mu_j=\ip b{C^j b}\ge0$. The leading eigenvalue $\lambda$
is the unique real solution greater than $a_0$ of
\begin{equation}\label{eq:scalar-root}
 \lambda=a_0+\sum_{j=0}^\infty\frac{\mu_j}{\lambda^{j+1}}.
\end{equation}
If $\lambda_N>a_0$ solves the equation truncated after $j=N$, then
\begin{equation}\label{eq:scalar-remainder}
 \lambda_N\uparrow\lambda,\qquad
 0\le\frac{\lambda-\lambda_N}{a_0}
 \le\frac{r_*^{N+4}}{3(1-r_*)}.
\end{equation}
In particular, the convergence rate does not deteriorate at
$\alpha=1/2$.
\end{proposition}
\begin{proof}
The block decomposition relative to $\mathbf1\oplus\mathbf1^\perp$
gives, first for finite compressions and then by Schatten convergence,
\[
 Q_3=a_0^3+3a_0\norm b^2+3\ip b{Cb}+\tr(C^3).
\]
All terms after $a_0^3$ are nonnegative. Consequently
$\norm C^3\le Q_3-a_0^3$ and
$\mu_0/a_0^2\le r_\alpha^3/3$.
The log-gamma series gives
\[
 \log\frac{Q_3}{a_0^3}
 =\sum_{n\ge2}\frac{\zeta(n)}n
       \bigl[(3/2)^n-1-(1/2)^n\bigr]\alpha^n.
\]
It has positive coefficients and converges at $\alpha=1/2$.
This proves $r_\alpha\le r_*$. Log-convexity at
$3/4,5/4,7/4$ gives
$\Gamma(1/4)/\Gamma(3/4)\le\sqrt{12}$, and $\pi>3$
then proves $Q_3/a_0^3<2$ at that endpoint.

The constant function is not an eigenfunction: $S_\alpha\mathbf1(u)$
tends to zero as $u\to\infty$ and is positive at zero.
Hence the variational principle gives $\lambda>a_0>\norm C$.
Taking the Schur complement gives
$\lambda=a_0+\ip b{(\lambda-C)^{-1}b}$.
Its convergent Neumann expansion is \eqref{eq:scalar-root}.
The right side decreases in $\lambda$, whereas the left side increases,
which gives uniqueness. For the truncated equation the derivative
of the left side minus the right side is at least one. Its omitted
tail is bounded, for $x\ge a_0$, by
\[
 \sum_{j>N}\frac{\mu_j}{x^{j+1}}
 \le\frac{\mu_0}{a_0}\frac{r_\alpha^{N+1}}{1-r_\alpha}.
\]
The preceding bound on $\mu_0$ proves \eqref{eq:scalar-remainder}.
\end{proof}

The coefficients can be written without eigenfunctions. Set
$M_0=1$ and, for $n\ge1$,
\[
 M_n=\ip{\mathbf1}{S_\alpha^n\mathbf1}
 =\int_{(0,\infty)^{n+1}}e^{-\sum_{j=0}^n u_j}
           \prod_{j=0}^{n-1}|u_j-u_{j+1}|^{-\alpha}\dd u_0\cdots\dd u_n.
\]
The resolvent block identity gives the finite recursion
\[
 \mu_0=M_2-a_0^2,\qquad
 \mu_j=M_{j+2}-a_0M_{j+1}
              -\sum_{k=0}^{j-1}\mu_kM_{j-k}\quad(j\ge1).
\]
These are path integrals, distinct from the cycle traces.
The scalar equation is an exact convergent representation; its
coefficients beyond the explicitly evaluated cases and its root are
not claimed to have a finite closed form.

\subsection{Explicit residual bounds and finite exponential trials}
The third trace also bounds the error of an arbitrary trial function,
without requiring a certified finite-dimensional eigensolver.

\begin{proposition}[Residual enclosure]\label{prop:residual}
Let $0<\alpha\le1/2$, $\norm h_{L^2(\nu)}=1$, and
\[
 \theta=\ip h{S_\alpha h},\qquad
 \sigma^2=\norm{S_\alpha h-\theta h}_{L^2(\nu)}^2,\qquad
 \beta=(Q_3-\theta^3)^{1/3}.
\]
If $\theta>\beta$, then
\begin{equation}\label{eq:residual-bound}
 \theta\le\lambda_{\max}(S_\alpha)
            \le\theta+\frac{\sigma^2}{\theta-\beta}.
\end{equation}
In particular, the separation condition holds whenever
$\theta\ge a_0$. All quantities in this enclosure have explicit
finite formulas when $h$ is a finite linear combination of
exponentials.
\end{proposition}
\begin{proof}
Write $\lambda_1\ge\lambda_2\ge\cdots\ge0$ for the spectrum.
Since $\theta\le\lambda_1$, the third trace implies
$\lambda_2\le(Q_3-\theta^3)^{1/3}=\beta$.
For each spectral value $x$, $(x-\lambda_1)(x-\beta)\ge0$.
Integration against the spectral probability measure of $h$ yields
$\sigma^2\ge(\lambda_1-\theta)(\theta-\beta)$, proving the claim.
The final separation statement follows from
$Q_3<2a_0^3$ in Proposition~\ref{prop:scalar-root}.
\end{proof}

Here are the finite formulas used to evaluate the residual. For
$b,c>0$, define
\begin{align}
 F_\alpha(b,c)
 &=\int_0^\infty\!\int_0^\infty
       e^{-br-ct}r^{-\alpha}(r+t)^{-\alpha}\dd r\dd t\notag\\
 &=\frac{\Gamma(2-2\alpha)}{1-\alpha}c^{2\alpha-2}
 {}_2F_1\left(2-2\alpha,1-\alpha;2-\alpha;1-\frac bc\right).
 \label{eq:gap-hypergeometric}
\end{align}
This follows by $r=xq$, $t=(1-x)q$, followed by the gamma integral.
At $\alpha=1/2$ the formula becomes elementary:
\begin{equation}\label{eq:gap-critical}
 F_{1/2}(b,c)=
 \begin{cases}
 \displaystyle\frac{2\arctan\sqrt{b/c-1}}{\sqrt{c(b-c)}},&b>c,\\[4pt]
 2/c,&b=c,\\[4pt]
 \displaystyle\frac{2\operatorname{artanh}\sqrt{1-b/c}}{\sqrt{c(c-b)}},&b<c.
 \end{cases}
\end{equation}
For $a,b,c>0$ put
\begin{align}
 H_\alpha(a,b,c)=\frac1{a+b+c}\Bigl\{&
 \Gamma(1-\alpha)^2\bigl[((b+c)c)^{\alpha-1}
                           +((a+b)a)^{\alpha-1}\bigr]\notag\\
 &+F_\alpha(a+c,c)+F_\alpha(a+c,a)
   +F_\alpha(b,b+c)+F_\alpha(b,a+b)\Bigr\}.
 \label{eq:triple-Laplace}
\end{align}
Splitting the three-coordinate integral into its six orderings gives
\[
 H_\alpha(a,b,c)=\iiint e^{-au-bv-cw}
                 |u-v|^{-\alpha}|v-w|^{-\alpha}\dd u\dd v\dd w.
\]
The common minimum integrates to $1/(a+b+c)$. When $v$ is the
middle coordinate the two gaps separate and give the two gamma
products; when it is an extreme coordinate they give the four
$F_\alpha$ terms. This also supplies an independent direct proof
of \eqref{eq:triple-Laplace}.

Choose distinct nonnegative rates $t_i$ and put $e_i(u)=e^{-t_i u}$.
Their Gram, first-operator and squared-operator matrices are
\begin{align}
 G_{ij}&=\frac1{1+t_i+t_j},\notag\\
 V_{ij}&=\frac{\Gamma(1-\alpha)
        [(1+t_i)^{\alpha-1}+(1+t_j)^{\alpha-1}]}{2+t_i+t_j},
       \label{eq:exponential-trial-matrices}\\
 W_{ij}&=H_\alpha(1+t_i,1,1+t_j).\notag
\end{align}
For any nonzero real coefficient vector $v$,
\[
 \theta=\frac{v^TVv}{v^TGv},\qquad
 \sigma^2=\frac{v^TWv}{v^TGv}-\theta^2.
\]
There is no omitted-basis contribution in this residual: $W$
represents $S_\alpha^2$, not the square of its finite compression.
A numerical eigensolver may select $v$; the final bound only uses
the chosen coefficients and these explicit quadratic forms.

\subsection{The common endpoint coefficient of the two parameter limits}
The same spectral problem controls the interaction of the weight
and memory boundaries. For fixed $d\le0$, let $C_{2,d}$ denote the limit of
$\Lambda_2(\alpha,d)/\sqrt{1-2\alpha}$ as $\alpha\uparrow1/2$;
its existence is included in the proof below.

\begin{theorem}[A common endpoint coefficient]\label{thm:joint-endpoint}
Define
\[
 C_{\exp}=\frac{\lambda_{\max}(S_{1/2})}{\sqrt2}.
\]
As $\alpha\uparrow1/2$ and $d\to-\infty$, jointly,
\begin{equation}\label{eq:joint-endpoint}
 \frac{\Lambda_2(\alpha,d)}{\sqrt{1-2\alpha}}
 =C_{\exp}+O\bigl(1-2\alpha+(1-d)^{-2}\bigr).
\end{equation}
The implied constant is uniform for $\alpha$ sufficiently close
to $1/2$ and $d\le0$. Moreover, putting $\kappa_0=1/2-d$,
\begin{equation}\label{eq:boundary-weight-prefactor}
 C_{2,d}=\frac{\lambda_{\max}(S_{1/2,\kappa_0})}{\sqrt2},\qquad
 0<C_{\exp}-C_{2,d}\le\frac{\sqrt3}{48\kappa_0^2}.
\end{equation}
The coefficient $C_{2,d}$ strictly increases as $d$ decreases.
In particular, both iterated limits of the left side of
\eqref{eq:joint-endpoint} equal $C_{\exp}$.
\end{theorem}
\begin{proof}
The strict positivity of the kernel makes the leading eigenvalue
simple. The operator-norm analyticity above therefore implies
\[
 \lambda_{\max}(S_\alpha)
       =\lambda_{\max}(S_{1/2})+O(1-2\alpha).
\]
The kernel-difference proof of \eqref{eq:kernel-rate} still applies
at $\alpha=1/2$, even though the separate kernels are not in
$L^2(\nu\otimes\nu)$ there. It gives, uniformly near that endpoint,
\[
 0\le\lambda_{\max}(S_\alpha)
       -\lambda_{\max}(S_{\alpha,\kappa})
 \le\frac{\alpha\sqrt{\Gamma(5-2\alpha)}}{24\kappa^2}.
\]
Here nonnegativity follows by testing on the positive principal
eigenfunction of the smaller kernel; the upper estimate follows
from the norm of the difference. Set $\varepsilon=1-2\alpha$.
The normalization estimate in the concentration proof gives
\[
 \varepsilon\mathcal Z_{2\alpha,\kappa}
 =\varepsilon\Gamma(\varepsilon)+O(\varepsilon\kappa^{-2})
 =1+O(\varepsilon),\qquad \kappa\ge1/2.
\]
Use \eqref{eq:concentration-exact}, with
$\kappa=1-d-\alpha\asymp1-d$, to prove
\eqref{eq:joint-endpoint}.

For fixed $d\le0$, $\kappa\to\kappa_0$. Continuity of the
hyperbolic-kernel operator, proved by the same local Schur bounds,
and the residue
$\varepsilon\mathcal Z_{1-\varepsilon,\kappa}\to1$
give the identity in \eqref{eq:boundary-weight-prefactor}.
At $\alpha=1/2$ the preceding difference bound, divided by
$\sqrt2$, is exactly $\sqrt3/(48\kappa_0^2)$.
Finally $[2\kappa\sinh(t/(2\kappa))]^{-1/2}$ strictly increases
with $\kappa$ for every $t>0$ and is strictly below $t^{-1/2}$.
Testing on positive principal eigenfunctions proves both strict
claims. This monotonicity concerns the boundary coefficients;
it does not assert monotonicity of the normalized constant at
a fixed interior memory parameter.
\end{proof}

The finite exponential trials give a reproducible interval evaluation
of this previously variational endpoint coefficient:
\begin{equation}\label{eq:critical-numerical-interval}
 \boxed{1.370323114331<C_{\exp}<1.370323114332.}
\end{equation}
The outward-rounded interval follows from
Proposition~\ref{prop:residual}, using 41 exponential functions,
the elementary formula \eqref{eq:gap-critical}, and interval
arithmetic as detailed in Appendix~\ref{app:residual-numerics}.
It is a numerical enclosure of an exact coefficient, not a claim
that this decimal interval is a closed-form evaluation.

\section{Scope and further questions}
The results concern Hermite rank two. The finite-weight spectral family,
its exponential-measure limit and the unweighted interval boundary are
different operators. In particular $C_{2,0}$ and $C_{\exp}$ are distinct
coefficients and their certificates must not be interchanged.
The auxiliary boundary operators do not assert the original stochastic
LIL at $\alpha=1/2$.

The spectral residual inequalities certify a selected trial once its full
operator residual is bounded. They do not require an exact eigenfunction
or exponential convergence of a chosen basis. The 25-decimal interval is
a rigorous numerical evaluation, not a finite closed form or proof of
transcendence. Neither unsuccessful symbolic recognition nor the lack of
a formula in previous numerical work proves nonexistence of a closed form.

An explicit remainder for the unweighted boundary asymptotic would quantify
its accuracy at a specified interior memory parameter. Further useful
questions concern improved finite-weight spectral bounds and evaluation
of the scalar path-moment equation. The proposed location of a maximum
in a Schur quotient must be proved before a grid observation is used as
a uniform inequality. The general-rank global bounds, local expansions
and concentration theorem are developed in \cite{MoldavskayaGeneral}.

\appendix
\section{A reproducible full-operator interval certificate}
\label{app:interval-certificate}
The model in this appendix is the single auxiliary operator
$K_*=K_{1/2}$ on $L^2([0,1],\dd x)$, corresponding to $m=2,d=0$.
All computations are deterministic.

\subsection{Exact trial functions and matrix elements}
Use the reflection-symmetric functions
\[
 \varphi_a(x)=x^a+(1-x)^a,\qquad
 a\in\{0,1/2,2,3/2,4,5/2,\ldots\}.
\]
Odd integer powers are redundant in this symmetric family.
For a trial $g=\sum_i c_i\varphi_{a_i}$, the coefficients $c_i$ are
fixed exact rationals with denominator $10^{100}$. The shift $\tau$
is an exact rational with denominator $10^{60}$ and $0<\tau<3$.
The numerical eigensolver selects the trial only; it is not used to
certify its approximation error.

The Gram and first-operator matrix elements are
\begin{align*}
 G_{ab}&=\frac2{a+b+1}+2B(a+1,b+1),\\
 V_{ab}&=2\frac{B(a+1,1/2)+B(b+1,1/2)}{a+b+3/2}+2J(a,b),\\
 J(a,b)&=\iint x^az^b|1-x-z|^{-1/2}\dd x\dd z.
\end{align*}
The direct term in $V$ follows by ordering the two variables; the
reflected term follows by $z=1-y$. The remaining moments have the finite
recursion, for $a>0$,
\[
 J(a,b)=\frac{aJ(a-1,b)+(a+1/2)^{-1}}{a+b+3/2},
 \qquad J(a,b)=J(b,a),
\]
with starting values
\[
 J(0,b)=2\left[B(b+1,3/2)+\frac1{b+3/2}\right],\qquad
 J(-1/2,-1/2)=2\pi+4\log2.
\]
For the recursion, integrate the divergence of
$x^az^b|1-x-z|^{-1/2}(x-1,z)$ on the two triangles separated by
$x+z=1$. Its divergence is
$[(a+b+3/2)x^az^b-ax^{a-1}z^b]|1-x-z|^{-1/2}$.
The only nonzero outer flux is at $z=1$, with integral
$(a+1/2)^{-1}$; the flux across a deleted strip about the singular line
tends to zero. The starting values follow by direct one-dimensional
integration. Half-integer beta values reduce by the gamma recurrence
to rational numbers and $\pi$. Thus no interval gamma evaluation is
required for these matrices.

\subsection{The exact residual and its endpoint logarithm}
Write $F_a(x)=K_*(y\mapsto y^a)(x)$. Direct integration and integration
by parts give
\[
 F_0(x)=2(\sqrt x+\sqrt{1-x}),\qquad
 F_{-1/2}(x)=\pi+2\log(1+\sqrt{1-x})-\log x,
\]
\[
 F_a(x)=\frac{2axF_{a-1}(x)+2\sqrt{1-x}}{2a+1},\qquad a>0.
\]
The integer powers start from $F_0$ and the positive half-integer
powers from $F_{-1/2}$. Reflection supplies $K_*\varphi_a$.
For example,
\[
 F_{1/2}(x)=\sqrt{1-x}
 +x\left[\frac\pi2+\log(1+\sqrt{1-x})-\frac12\log x\right].
\]
This identity shows why a power-only endpoint argument is insufficient
to establish exponential convergence of a trial basis. No full
eigenfunction expansion is assumed here.

On $0<x<1/2$ the exact shifted residual has the form
\[
 r(x):=K_*g(x)-\tau g(x)=E(x)+\sqrt x\,O(x)+B(x)\log x,
\]
where $E,O$ are analytic in $|z|<1$ and $B$ is a finite polynomial.
The Taylor coefficients follow from the preceding recursions and
\begin{align*}
 \sqrt{1-z}&=1-\sum_{k\ge1}
       \frac{\binom{2k}{k}}{4^k(2k-1)}z^k,\\
 \pi+2\log(1+\sqrt{1-z})
 &=\pi+2\log2-\sum_{k\ge1}\frac{\binom{2k}{k}}{k4^k}z^k,\\
 F_{-1/2}(1-x)&=\pi+2\sqrt x\sum_{k\ge0}\frac{x^k}{2k+1}.
\end{align*}
Reflection extends $r$ from the half interval to the full interval.

\subsection{An analytic bound for the omitted series}
Choose an exactly checked integer $H\ge\sum_i|c_i|$, put
$k_0=\max_i\lfloor a_i\rfloor$, and set
\[
 M_0=32H\,2^{k_0+1}.
\]
Both analytic parts $E,O$ have modulus at most $M_0$ on $|z|=3/4$.
Here are explicit bounds that verify the constant. The initial square-root
series has modulus at most $\sqrt{7/4}<3/2$; the direct logarithmic
analytic part is bounded by $\pi+2\log2+\log4<6$; and the reflected
odd series is bounded by $2/(1-3/4)=8$. In the direct recurrence the
multiplier has modulus at most $3/4$ and the additive term at most 3,
giving bounds 12 and 2 for its even and odd parts. In the reflected
recurrence the multiplier is at most $7/4$; induction bounds the
integer even/odd parts by $3(7/4)^{k_0}$ and $5(7/4)^{k_0}$ and the
half-integer parts by $4(7/4)^{k_0+1}$ and $11(7/4)^{k_0+1}$.
The terms $\tau g$, with $|\tau|<3$, are covered by the displayed
factor 32 and $(7/4)^{k_0+1}<2^{k_0+1}$.

Let $E_L,O_L$ be the degree-$L$ Taylor polynomials and
$r_L=E_L+\sqrt x\,O_L+B\log x$, extended by reflection.
Cauchy's coefficient estimate gives a pointwise tail at most
$6M_0(4x/3)^{L+1}$ on $[0,1/2]$. Integrating its square on the two
half intervals proves
\begin{equation}\label{eq:series-tail}
 \|r-r_L\|_2\le T_L:=
 \frac{6M_0}{\sqrt{2L+3}}\left(\frac23\right)^{L+1}.
\end{equation}
All terms of $\|r_L\|_2^2$ are finite combinations of
\begin{align*}
 \int_0^{1/2}x^s\dd x&=\frac{2^{-s-1}}{s+1},\\
 \int_0^{1/2}x^s\log x\dd x
 &=\frac{2^{-s-1}}{s+1}\left(-\log2-\frac1{s+1}\right),\\
 \int_0^{1/2}x^s\log^2x\dd x
 &=\frac{2^{-s-1}}{s+1}
 \left[(\log2)^2+\frac{2\log2}{s+1}+\frac2{(s+1)^2}\right],
 \qquad s>-1.
\end{align*}
If directed interval arithmetic bounds this finite norm squared from
above by $A_L$, the full residual obeys
\[
 \sigma^2\le\frac{\|r\|_2^2}{\|g\|_2^2}
 \le\frac{(\sqrt{A_L}+T_L)^2}{\|g\|_2^2}.
\]
This controls $K_*g$ itself, including its component outside the trial
space. It is not the residual of a finite matrix compression.

\subsection{Certificate data}
The exact rational trial has dimension 128. With $k_0=126$,
$H=10^{239}$ and $L=1825$, the preceding finite-moment calculation
and analytic remainder bound give
\[
 T_L<1.550\cdot10^{-44},\qquad \sigma^2<1.166\cdot10^{-26}.
\]
All finite calculations are enclosed by outward interval arithmetic.
The exact rational trial and certificate data are included
in the source file of this article. Together with the proof above,
they permit independent verification of Theorem~\ref{thm:interval-value}.

\section{Additional deterministic computations}\label{app:numerics}
An independent check evaluates the gamma product by ordering $0<x<y<z<1$ and setting
$z=r$, $y=rv$, $x=rvu$. After exact radial integration, product
Gauss--Jacobi quadrature is applied to
\[
 \frac6{3-3d-3\alpha}\int_0^1\int_0^1
 u^{-d}(1-u)^{-\alpha}v^{1-2d-\alpha}(1-v)^{-\alpha}
 (1-uv)^{-\alpha}\dd u\dd v.
\]
There are 12 deterministic evaluations with $(\alpha,d)$ equal to
$(0.1,0)$, $(0.2,0.1)$, or $(0.3,-0.25)$ and $N=64,128,256,512$
nonuniform nodes in each variable. No single grid spacing applies.
There are zero Monte Carlo runs and no seed. Relative differences at
$N=512$ are approximately $4.11\cdot10^{-11}$, $2.73\cdot10^{-9}$ and
$2.82\cdot10^{-7}$, respectively. They decrease under refinement and are
corroboration of the exact identity, not rigorous numerical error bounds.

\subsection{Laguerre matrix computations}
The spectral calculation uses
\[
 \alpha\in\{0.1,0.2,0.4\},\qquad
 \kappa\in\{1,4,16,\infty\},\qquad N\in\{16,64,128\}.
\]
At infinity the matrix recurrence uses only the generalized binomial
coefficients of \eqref{eq:Laguerre-limit}. At finite $\kappa$ the
coefficient correction is evaluated as
\[
 q_n^{(\kappa)}-q_n^{(\infty)}
 =\int_0^\infty e^{-t}t^{-\alpha}
 \left[\left(\frac{\sinh(t/(2\kappa))}{t/(2\kappa)}\right)^{-\alpha}-1\right]
 [\mathcal L_n(t)-\mathcal L_{n-1}(t)]\dd t,
\]
with $\mathcal L_{-1}=0$ and quadrature orders 128 and 256.
The constant coefficient is then replaced by its exact beta formula.
This gives 54 finite-weight eigensolves and nine limiting eigensolves.
The largest change in either spectral endpoint when doubling the
quadrature order is $4.441\cdot10^{-16}$. The first coefficient also
agrees with the independent beta derivative within
$1.477\cdot10^{-14}$. These small differences do not constitute
interval-arithmetic certification.

\begin{table}[htbp]
\centering
\caption{Floating-point evaluations of the spectral endpoints in
\eqref{eq:spectral-enclosure}, divided by $\sqrt{2Q_{2,\kappa}}$,
at $N=128$. The exact endpoint formulas are rigorous bounds;
the displayed decimals have not been certified against rounding error.}
\label{tab:Laguerre-enclosures}
\begin{tabular}{rrrr}
\toprule
$\alpha$&$\kappa$&Lower endpoint&Upper endpoint\\
\midrule
0.1&1&0.702051001&0.702053930\\
0.1&$\infty$&0.702322594&0.702325473\\
0.2&1&0.681784548&0.681844929\\
0.2&4&0.682693750&0.682752778\\
0.2&16&0.682754183&0.682813118\\
0.2&$\infty$&0.682758232&0.682817160\\
0.4&1&0.516090801&0.518823588\\
0.4&$\infty$&0.518518663&0.521197758\\
\bottomrule
\end{tabular}
\end{table}

An independent exact rational calculation at $\alpha=1/5$ verifies
the leading $4\times4$ limiting matrix after removing its common
$\Gamma(1-\alpha)$ factor. It expands the Laguerre polynomials and
integrates monomials over the two ordered triangles; it does not use
the generating recurrence. Four independent third-trace checks at
$\alpha=0.1,0.2,0.4,0.49$ use the integral
\[
 2\Gamma(2-3\alpha)\int_0^1
 \frac{t^{-\alpha}(1-t)^{-\alpha}}{(1+t)^{2-3\alpha}}\dd t.
\]
Adaptive algebraically weighted quadrature with absolute and relative
tolerances $2\cdot10^{-12}$ and subdivision limit 250 agrees with
$Q_{3,\infty}$ within $1.111\cdot10^{-15}$ in relative terms.

All computations are deterministic: Monte Carlo runs $=0$, with no
random seed. No numerical assertion is used as
a substitute for a proof of the variational, asymptotic or algebraic
identities in the main text.

\subsection{Interval evaluation of exponential-limit constants}
\label{app:residual-numerics}
The exponential-trial calculation implements
Proposition~\ref{prop:residual} independently of the Laguerre
compression calculation. They concern the deterministic rank-two
exponential-measure limit; the row $\alpha=1/2$ evaluates
$C_{\exp}$, not a stochastic LIL constant at an inadmissible
critical parameter. There are no simulated paths, Monte Carlo
replications, or random seeds.

For each $\alpha\in\{1/10,1/5,2/5,1/2\}$, use
\[
 \{t_i\}=\{0\}\cup\{2^j:-4\le j\le J\},\qquad J\in\{4,8,12\}.
\]
These give 12 calculations with dimensions $10,14,18$.
A further critical calculation uses the 41 distinct rates
\[
 \{0\}\cup\{2^j,(3/2)2^j:-3\le j\le16\}.
\]
These are nonuniform rates of exponential functions, not a spatial
discretization grid. The Gram matrix is treated explicitly. A
Cholesky reduction and an ordinary eigensolver select a trial vector;
each coefficient is then rounded to an exact rational with
denominator $10^{50}$. The certification uses that rational vector,
not the eigensolver's claimed eigenvalue or residual.

All subsequent quantities use 80-decimal-digit outward interval
arithmetic for elementary operations. Gamma values are enclosed
by shifting their argument by 128 and retaining 40 terms of the
log-gamma Stirling expansion. For positive real arguments the
remainder lies between zero and the first omitted term, with its
sign \cite[\S5.11(ii)]{DLMF}. Negative arguments in $(-1,0)$ are
first shifted by the gamma recurrence. This implementation does
not call an interval gamma or interval hypergeometric function.

Away from the critical row, the standard hypergeometric
transformations \cite[Eqs.~15.8.1 and 15.8.4]{DLMF} reduce each
series argument to $z\in[0,1/2]$, with positive parameters $a,b,c$.
For the positive terms $T_n=(a)_n(b)_nz^n/((c)_nn!)$,
the tail after $T_n$ is bounded by $T_{n+1}/(1-\rho_n)$ whenever
\[
 \rho_n=z\left(1+\frac{(a-1)_+}{n+2}\right)
             \left(1+\frac{(b-c)_+}{n+1+c}\right)<1.
\]
Indeed $\rho_n$ bounds every ratio after $T_{n+1}$.
The stopping threshold for this explicit absolute tail bound is
$10^{-70}$. Formula \eqref{eq:gap-critical} replaces these series
by elementary functions at $\alpha=1/2$.

\begin{table}[htbp]
\centering\small
\begin{tabular}{ccccc}
\toprule
$\alpha$ & quantity & dimension & lower endpoint & upper endpoint\\
\midrule
$1/10$ & $\Lambda_2^{\exp}(1/10)$ &18&0.702322596879&0.702322596885\\
$1/5$ & $\Lambda_2^{\exp}(1/5)$ &18&0.682758258367&0.682758258398\\
$2/5$ & $\Lambda_2^{\exp}(2/5)$ &18&0.518519027222&0.518519027647\\
$1/2$ & $C_{\exp}$ &18&1.370323112119&1.370323117614\\
$1/2$ & $C_{\exp}$ &41&1.370323114331&1.370323114332\\
\bottomrule
\end{tabular}
\caption{Outward-rounded residual enclosures.
The final row determines the joint-limit coefficient.}
\label{tab:residual-certified}
\end{table}

The unrounded width in the final row is below $2.951\cdot10^{-13}$;
its squared-residual upper bound is below $2.815\cdot10^{-13}$.
The exact rational coefficients are included in the source
file of this article.
The historical floating-point intervals in
Table~\ref{tab:Laguerre-enclosures} keep their original,
uncertified-decimal status.


\begin{thebibliography}{99}
\bibitem{LP25} N. N. Leonenko and A. Pepelyshev (2025).
Numerical computation of the Rosenblatt distribution and applications.
\href{https://arxiv.org/abs/2506.23337}{arXiv:2506.23337}.
\bibitem{MoldavskayaGeneral}
E. Moldavskaya (2026a).
Global bounds and parameter expansions for weighted LIL constants.
Unpublished manuscript.

\bibitem{MoldavskayaLIL}
E. Moldavskaya (2026b).
Law of the Iterated Logarithm for Weighted Sums of Functionals
of Long-Memory Gaussian Sequences.
Preprint.
\href{https://arxiv.org/abs/2606.21006}{arXiv:2606.21006}.

\bibitem{MoldavskayaML}
E. Moldavskaya (2026c).
Weighted Empirical Risk Minimization for Machine Learning under
Long-Range Dependence: Exact Pathwise Rates and Learning-Error Geometry.
Preprint.
\href{https://arxiv.org/abs/2609.10767}{arXiv:2609.10767}.\bibitem{MO86} T. Mori and H. Oodaira (1986).
The law of the iterated logarithm for self-similar processes represented
by multiple Wiener integrals.
\emph{Probability Theory and Related Fields} \textbf{71}, 367--391.
\href{https://doi.org/10.1007/BF01000212}{doi:10.1007/BF01000212}.
\bibitem{MO87} T. Mori and H. Oodaira (1987).
The functional iterated logarithm law for stochastic processes represented
by multiple Wiener integrals.
\emph{Probability Theory and Related Fields} \textbf{76}, 299--310.
\href{https://doi.org/10.1007/BF01297487}{doi:10.1007/BF01297487}.
\bibitem{DLMF} NIST Digital Library of Mathematical Functions,
Sections 5.6, 5.7, 5.11, 5.12, 5.14, 5.15 and 18.12; in particular
Gautschi's inequality, Eq.~5.6.4, the Selberg integrals, Eqs.~5.14.4--5,
and the Laguerre generating function, Eq.~18.12.13.
\url{https://dlmf.nist.gov/}.
\bibitem{VT13} M. S. Veillette and M. S. Taqqu (2013).
Properties and numerical evaluation of the Rosenblatt distribution.
\emph{Bernoulli} \textbf{19}(3), 982--1005.
\href{https://arxiv.org/abs/1307.5990}{arXiv:1307.5990}.

\end{thebibliography}
\end{document}